\documentclass{amsart}

\usepackage{amsmath, amssymb, amsthm}

\theoremstyle{plain}
\newtheorem{thm}{Theorem}

\newtheorem{lem}[thm]{Lemma}
\newtheorem{prop}[thm]{Proposition}

\begin{document}

\title[Loewy lengths of centers of blocks]{Loewy lengths of centers of blocks and exponents of defect groups}
\author[Y. Otokita]{Yoshihiro Otokita}
\address{Yoshihiro Otokita: \newline Department of Mathematics and Informatics \newline
Graduate School of Science \newline Chiba University \newline
Chiba 263--8522 Japan}

\maketitle

\begin{abstract}
In this paper we study the Loewy structure of the center $ZB$ of a block of a finite group with respect to an algebraically closed field of prime characteristic. We first state a new method for calculating the Loewy length $LL(ZB)$ of $ZB$ in terms of subsections and lower defect groups. By applying this result we give an upper bound of $LL(ZB)$ which depends on the exponent of a defect group of $B$ for some cases.
\end{abstract}

\section*{Introduction}
The purpose of this paper is to study new methods for calculating the Loewy length of the center of a block of a finite group. We then extend some results in K\"{u}lshammer-Sambale \cite{KS}, and \cite{KOS}.

In the following, let $G$ be a finite group and $F$ an algebraically closed field of characteristic $p>0$. For a block $B$ of the group algebra $FG$, we denote by $ZB$ its center. In modular representation theory, various problems and previous works indicate the relations between $ZB$ and a defect group $D$ of $B$. For example, Brauer {\cite[Problem 20]{Bra1}} conjectured that the number of irreducible ordinary characters associated to $B$, which equals the dimension of $ZB$ over $F$, is bounded above by the order $p^{d}$ of $D$. In particular, there is a possibility that the Loewy structure of $ZB$ is connected with the exponent $p^{e}$ of $D$. For instance, the following results are known:
\begin{itemize}
\item K\"{u}lshammer \cite{K1} proved that $z^{p^{e}}=0$ for any element $z$ in the Jacobson radical of $ZB$. 
\item If $D$ is non-abelian with cyclic maximal subgroup (i.e. $e=d-1$), then
\begin{equation*}
LL(ZB) \le \begin{cases}
                      2^{e-1}+1 & (p=2) \\
                      p^{e-1}     & (p \neq 2)
                 \end{cases}
\end{equation*}    
by {\cite[Proposition 7]{KOS}}, where $LL(ZB)$ is the Loewy length of $ZB$. In general,
\[ LL(ZB) \le \left(\frac{d}{e} + 1\right) \left(\frac{d}{2} + \frac{1}{p-1} \right) (p^{e}-1) \] 
provided $d \ge 1$ by {\cite[Theorem 6]{KOS}}. 
\item If $D$ is abelian, then
\[ \frac{p^{e}+p-2}{p-1} \le LL(ZB) \]
by {\cite[Theorem]{Ot1}} \footnote{This article is a short note that proves Theorem 3 in the preprint below: \\
\ \ \ \ \ \ Y. Otokita, \textit{Notes on Loewy series of centers of $p$-blocks}, arXiv 1806.08029} 
\end{itemize}
Motivated by the facts above, the present paper consider an upper bound on $LL(ZB)$. In the next section (the main part of this paper), we define a function $\rho(d, e)$ which depends more on $e$ than on $d$ and prove that $LL(ZB)$ is bounded above by $\rho(d, e)$ for two cases. Moreover, the third section deals with general bounds for blocks with non-abelian defect groups and we improve a result in \cite{KOS}.

Throughout this paper we use the following notations.

The sum of all elements in a subset $U$ of $G$ is denoted by $U^{+}$. For $a \in FG$ and $g \in G$, we denote by $a(g)$ the coefficient of $g$ in $a$. Moreover, $o(g)$ is the order of $g$. For a finite--dimensional algebra $\Lambda$ over $F$ with Jacobson radical $J(\Lambda)$, its Loewy length $LL(\Lambda)$ is the smallest positive integer $L$ such that $J(\Lambda)^{L}=\{0\}$. Finally, $R\Lambda$ denotes the Reynolds ideal, which is the intersection of the center and the socle of $\Lambda$. 

\section*{Main problem}
First of all, we formulate our problem in this section.

For two non-negative integers $m$ and $n$, we define
\begin{equation*}
\rho (m, n) = \begin{cases}
                           q(p^{n}-1)+p^{r} & (n \neq 0) \\
                           \ \ \ \ \ \ \ \ 1         &      (n=0)
                       \end{cases}
\end{equation*}
where $q$ and $r$ are integers such that $m=qn+r, 0 \le r \le n-1$. Then $\rho$ is a monotone increasing function as follows.

\begin{prop} \label{prop1} Let $m, n, k$ and $l$ are non-negative integers. Then the following hold:
\begin{enumerate}
\item If $m \le k$ and $n \le l$, then $\rho(m, n) \le \rho(k, l)$ with equality if and only if one of the following holds:
\begin{enumerate}
\item $l=0$;
\item $n=l < m=k$;
\item $m=k \le n$.
\end{enumerate}
\item $1 \le \rho(m, n) \le p^{m}$.
\end{enumerate}
\end{prop}
\begin{proof}
It is clear that $\rho(m, n) = p^{m}$ if $m \le n$, and $p^{n} < \rho(m, n)$ if $1 \le n<m$. Hence (1) implies (2) and it suffices to prove that $\rho(m, n) < \rho(m+1, n)$ and $\rho(m, n) < \rho(m, n+1)$ for $m > n \ge 1$. We put $m=qn+r$ where $q \ge 1, 0 \le r \le n-1$. Then
\begin{equation*}
m+1 = \begin{cases}
               qn+(r+1) & (0 \le r \le n-2) \\
               (q+1)n+0 & (r=n-1).
            \end{cases}
\end{equation*}
Thus 
\begin{equation*}
\rho(m+1, n) = \begin{cases}
                        q(p^{n}-1)+p^{r+1} & (0 \le r \le n-2) \\
                        (q+1)(p^{n}-1) + 1 = q(p^{n}-1) + p^{n} & (r=n-1)
                      \end{cases}
\end{equation*}
is bounded below by $\rho(m, n)$. We next set $m=s(n+1)+t$, where $s \ge 1, 0 \le t \le n$. Then
\[ q = s\frac{n+1}{n} + \frac{t-r}{n} \le 2s + \frac{t-r}{n} \le 2s+1. \]
If $q=2s+1$, then $n=t=1$ and $r=0$. Hence 
\[ \rho(m, n) = (2s+1)(p-1)+1 < s(p^{2}-1)+p = \rho(m, n+1). \]
If $t-r < 0$, then $q \le 2s + (t-r)/n < 2s$. Therefore we may assume that $q=2s$ and $t-r \ge 0$, or $q \le 2s-1$. In the first case, 
\[ \rho(m, n) = 2s(p^{n}-1)+p^{r} < s(p^{n+1}-1) + p^{t} = \rho(m, n+1). \]
In the second case, 
\[ \rho(m, n) \le (2s-1)(p^{n}-1)+p^{r} < s(p^{n+1}-1)+1 \le \rho(m, n+1)\]
as claimed. 
\end{proof}

By Proposition \ref{prop1}, $1 \le \rho(d, e) \le p^{d}$ for a block $B$ with defect group $D$. In \cite{KS}, it was proved that $LL(ZB)$ is bounded above by $\rho(d, e)$ for two cases (see Theorem \ref{thm2} below). This inequality is a refinement of Okuyama \cite{Ok1}, which states $LL(ZB) \le p^{d}$, and characterize a connection between the Loewy structure of $ZB$ and $p^{e}$ since $\rho(d, e)$ depends more on $e$ than on $d$. 

\begin{thm} \label{thm2} {\cite[Theorem 1, Proposition 10 and 13]{KS}} Let $B$ be a block of $FG$ with defect group $D$ of order $p^{d}$ and exponent $p^{e}$. Then the following hold:
\begin{enumerate}
\item If $D$ is abelian, then $LL(ZB) \le \rho(d, e)$. If moreover $D$ is normal in $G$, then equality occurs if and only if $B$ is nilpotent and $D$ has type $(e, \dots, e, f)$ for some $e \ge f$.
\item If $p=2$ and $D$ is non-abelian of order 8 or 16, then $LL(ZB) < \rho(d, e)$.
\end{enumerate}
\end{thm}
\begin{proof}
Suppose that  $D$ is abelian of type $(e_{1}, \dots, e_{r})$, where $e=e_{1} \ge \dots \ge e_{r}$. Then we can prove  
\begin{align*}
LL(FD) = p^{e_{1}} + \cdots + p^{e_{r}}-r+1 & = p^{e}-1 + \{p^{e_{2}} + \dots + p^{e_{r}} - (r-1) +1\} \\
                                                                     & \le p^{e}-1+ \rho(d-e, e_{2}) \\
                                                                     & \le p^{e}-1 + \rho(d-e, e) = \rho(d, e),
\end{align*}
and $LL(FD)=\rho(d, e)$ if and only if $e=e_{2}= \dots =e_{r-1}$ by induction on $r$ and Proposition \ref{prop1}. Hence the claim in (1) follows from {\cite[Theorem 1 and Proposition 7]{KS}}. The other cases are due to {\cite[Proposition 10 and 13]{KS}}.
\end{proof}

In Theorem \ref{thm2} (1), $\rho(d, e)$ is also an upper bound for the Loewy length of $FD$ as mentioned in its proof. On the other hand, the extra--special $p$-group $p^{1+2}_{+}$ of order $p^{3}$ and exponent $p$ satisfies $\rho(3, 1)=3p-2$ and $LL(Fp^{1+2}_{+}) = 4p-3$. Thus the relations between these values are unclear in general.

In the following we prove the same inequality as in Theorem \ref{thm2} for other two cases. For this we need some notations and lemmas. 

Let $\mathcal{F}$ be the fusion system of $B$. Following {\cite[Lemma 1.36]{S1}}, we fix a complete set $\mathcal{S}$ of representatives for $G$-conjugacy classes of $B$-subsections such that $u \in D$ and $b$ has defect group $C_{D}(u)$ for any $(u, b) \in \mathcal{S}$. Moreover, we put 
\[ \mathcal{S}_{0} = \{ (u, b) \in \mathcal{S} \mid u \in Z(D) \} \ \ \ \text{(the set of all major subsections in $\mathcal{S}$)}. \]

For a $p$-subgroup $P$ of $G$, we denote by $Z_{\le P}(FG)$ the $F$-linear span of all the elements $C^{+}$ such that $C$ is a conjugacy class of $G$ and some defect group of $C$ is contained in a $G$-conjugate of $P$. Furthermore, we define
\[ Z_{< P}(FG) = \sum_{Q < P} Z_{\le Q}(FG). \]

For a block $B$ of $FG$ with block idempotent $\varepsilon$, we set
\begin{align*}
& Z_{\le P}(B) = Z_{\le P}(FG) \varepsilon, \\
&Z_{< P}(B) = Z_{< P}(FG) \varepsilon. 
\end{align*}
These are ideals of $ZB$. 

We here prove a lemma on controlled blocks. A block $B$ is said to be \textit{controlled} if all morphisms of $\mathcal{F}$ are generated by restrictions from ${\rm Aut}_{\mathcal{F}}(D)$. In this case $\mathcal{S} \backslash \mathcal{S}_{0}$ is not an empty set unless $D$ is abelian.

\begin{lem} \label{lem3} Let $B$ be a controlled block of $FG$ with non-abelian defect group $D$. Then $J(ZB)^{M} \subseteq Z_{\le Z(D)}(B)$, where
\[M = \max \{LL(Zb) \mid (u, b) \in \mathcal{S} \backslash \mathcal{S}_{0} \}. \]
\end{lem}

\begin{proof}
Take an element $z$ in $J(ZB)^{M}$, $g \in G$ and a defect group $Q$ of $g$. If $z(g) \neq 0$, then $Q^{h} \subseteq D$ for some $h \in G$. We now suppose $x \in Q^{h} \backslash Z(D)$ and $\sigma : ZFG \to ZFC_{G}(x)$ denotes the Brauer homomorphism. For all Brauer correspondent $\beta$ of $B$ in $C_{G}(x)$, $LL(Z\beta) \le M$ by our assumption. Since $g^{h} \in C_{G}(x)$, we have $z(g)=z(g^{h})=z^{\sigma}(g^{h})=0$, a contradiction. Therefore $Q \le_{G} Z(D)$ provided $z(g) \neq 0$ and thus $J(ZB)^{M} \subseteq Z_{\le Z(D)}(FG) \cap ZB = Z_{\le Z(D)}(B)$. 
\end{proof}

In the following we use the notations and terminologies on lower defect groups in Feit \cite{F1}. It is known that the number of irreducible Brauer characters associated to $B$ equals the sum of $m^{1}_{B}(P)$, where $P$ ranges over the representatives for conjugacy classes of $p$-subgroups of $G$, and $m^{1}_{B}(D)=1$. Furthermore, we define the \textit{hyperfocal subgroup} of $B$ as follows:
\[ \mathfrak{hyp}(B) = < f(x)x^{-1} \mid x \in Q \le D, f \in O^{p}({\rm Aut}_{\mathcal{F}}(Q)) >. \]

\begin{prop} \label{prop4} Let $B$ be a block of $FG$ with defect group $D$. If $\mathfrak{hyp}(B)$ is cyclic, then one of the following holds: 
\begin{enumerate}
\item $D$ is abelian and 
\[ LL(ZB) = LL(F[D/\mathfrak{hyp}(B)]) + \frac{|\mathfrak{hyp}(B)|-1}{|I(B)|}, \]
where $I(B)$ is the inertial quotient of $B$.
\item $D$ is non-abelian and
\[ LL(ZB) \le \max\{ LL(Zb) \mid (u, b) \in \mathcal{S} \backslash \mathcal{S}_{0} \}. \]
\end{enumerate}
\end{prop}
\begin{proof}
Put $I=I(B)$ and $R=C_{D}(I)$. We first suppose $D$ is abelian. Then $\mathfrak{hyp}(B)=[D, I]$ and $D = \mathfrak{hyp}(B) \times R$. If $\beta$ denotes a Brauer correspondent of $B$ in $C_{G}(R)$, then it holds from {\cite[Theorem 2]{W2}} that $ZB \simeq Z\beta$. Let $\bar{\beta}$ be a block of $C_{G}(R)/R$ dominated by $\beta$. By {\cite[Theorem 7]{KOW}}, a source algebra of $\beta$ is isomorphic to a tensor product of $FR$ and a source algebra of $\bar{\beta}$. Since a block is Morita equivalent to its source algebra, we have $Z\beta \simeq FR \otimes Z\bar{\beta}$. Since $\bar{\beta}$ has defect group $D/R \simeq \mathfrak{hyp}(B)$ and inertial quotient $I(\bar{\beta}) \simeq I(B)$, 
\[ LL(ZB) = LL(FR) + LL(Z\bar{\beta}) - 1 = LL(F[D/\mathfrak{hyp}(B)]) + \frac{|\mathfrak{hyp}(B)|-1}{|I(B)|} \]
by {\cite[Corollary 2.8]{KKS}}. We next suppose $D$ is non-abelian. $B$ is a controlled block as mentioned in {\cite[Theorem 3]{W1}}. Thus we need only to prove $Z_{\le Z(D)}(B) = \{0\}$; otherwise $m^{\pi}_{B}(P) \neq 0$ for  a $p$-element $\pi \in G$ and a subgroup $P \le Z(D)$ by Lemma 10.8 in \cite{F1}. Furthermore, we can replace $\pi$ by an element in $Z(D)$ (see {\cite[Corollary 5.7]{Ol1}}). Hence $m^{1}_{b}(Q) \neq 0$ for some $(u, b) \in \mathcal{S}_{0}$ and $Q \le_{G} Z(D)$ by {\cite[Corollary 7.7]{Ol1}} (cf. {\cite[Theorem 10.10]{F1}}). We here remark that $b$ has defect group $D$, inertial quotient $I(b) \simeq C_{I}(u)$ and  cyclic hyperfocal subgroup (see {\cite[Lemma 6]{W1}}).  If $B$ is nilpotent (i.e. $|\mathfrak{hyp}(B)|=1$), then so is $b$. Since a nilpotent block has exactly one irreducible Brauer character, this contradicts $m^{1}_{b}(Q) \neq 0$. If $B$ is non-nilpotent, then $D = \mathfrak{hyp}(B) \rtimes R$ by {\cite[Lemma 4]{W1}}. In this case, $Q$ is a $C_{G}(u)$-conjugate of $C_{D}(I(b))$ and thus $R \le_{G} Z(D)$ (see {\cite[Lemma 10 or the proof of Theorem 1]{W1}}). Thereby $R \le Z(D)$ as $B$ is controlled, but this contradicts our assumption that $D$ is non-abelian. We have thus proved (2). 
\end{proof}

The main theorem in this section is the following.

\begin{thm} \label{thm5} Let $B$ be a block of $FG$ with defect group $D$ of order $p^{d}$ and exponent $p^{e}$. Then the following hold:
\begin{enumerate}
\item If $\mathfrak{hyp}(B)$ is cyclic, then $LL(ZB) \le \rho(d, e)$ with equality if and only if $B$ is nilpotent and $D$ is abelian of type $(e, \dots, e, f)$ for some $e \ge f$;
\item If $D$ is non-abelian metacyclic, then $LL(ZB) < \rho(d, e)$.
\end{enumerate}
\end{thm}
\begin{proof}
We first prove (1). If $D$ is abelian, then the claim follows from Theorem \ref{thm2}, its proof and Proposition \ref{prop4}. If $D$ is non-abelian, then there exists $(u, b) \in \mathcal{S} \backslash \mathcal{S}_{0}$ such that $LL(ZB) \le LL(Zb)$. Since $b$ has cyclic hyperfocal subgroup and defect group $C_{D}(u) < D$, we obtain $LL(ZB) \le \rho(d-1, e) < \rho(d, e)$ by induction on $|D|$ and Proposition \ref{prop1}. 

In order to prove (2), we may assume $B$ is non-nilpotent. Then $p=2$ and $e=d-1$, or $p \neq 2$ and $D$ is a split extension of two cyclic groups ({\cite[Theorem 8.1 and 8.8]{S1}}). The first case is due to {\cite[Proposition 7]{KOS}}. In the second case, $B$ has cyclic hyperfocal subgroup (see the proof of {\cite[Theorem 8.8]{S1}}. Note that $\mathfrak{hyp}(B)$ is contained in the focal subgroup of $B$). Thus (2) follows from (1).
\end{proof}

Recently, Tasaka and Watanabe \cite{TW} proved that a block with non-abelian metacyclic defect group of odd order is isotypic to its Brauer correspondent in $N_{G}(D)$. It is well-known that isotypy between two blocks induces isomorphism between their centers. As $B$ has cyclic inertial quotient $I$ in this case, $ZB$ is isomorphic to $ZF[D \rtimes I]$ by K\"{u}lshammer \cite{K2}.

\section*{General bounds for non-abelian defect groups}
In this section we improve {\cite[Theorem 3]{KOS}}, which states
\begin{equation*}
 LL(ZB) < \begin{cases}
                     p^{d-1} & (p=2, 3) \\
                     4p^{d-2} & (p \ge 5)
                  \end{cases}
 \end{equation*}
for a block $B$ with non-abelian defect group $D$. 

As in the preceding section, let $\mathcal{S}$ be a complete set of representatives for $G$-conjugacy classes of $B$-subsections. For each $(u, b) \in \mathcal{S}$, $b$ dominates a block $\bar{b}$ of $C_{G}(u)/\left<u\right>$ with defect group $C_{D}(u)/\left<u\right>$. 

We now prepare two lemmas. The first one refines the second part of Proposition 2 in \cite{KS}. 

\begin{lem} \label{lem6} Let $B$ be a block of $FG$ with non-trivial defect group $D$. Then the following hold: 
\begin{enumerate}
\item $LL(ZB/RB) \le \max \{ LL(Zb/Rb) \mid (u, b) \in \mathcal{S}, o(u)=p \}$;
\item For any $(u, b) \in \mathcal{S}$, $LL(Zb) \le p^{n} \times LL(Z\bar{b})$ where $p^{n}=o(u)$.
\end{enumerate}
\end{lem}
\begin{proof}
We first prove (1). Let $\varepsilon$ be the block idempotent of $B$, $m$ the right hand side of the inequality and fix an element $z$ in $J(ZB)^{m}$. Since $RFG$ has an $F$-basis consisting of the sums $S^{+}$, where $S$ runs through the $p'$-sections of $G$, and $RB=RFG\varepsilon$, we need only to prove $z(xy)=z(y)$ for all non-trivial $p$-element $x$ in $G$ and $p'$-element $y$ in $C_{G}(x)$. We here put $v=x^{p^{r-1}}$ where $p^{r}=o(x)$ and $\sigma : ZFG \to ZFC_{G}(v)$ denotes the Brauer homomorphism. Remark that $x, y, xy \in C_{G}(v)$. If $v \notin_{G} D$, then $\varepsilon^{\sigma}=0$ and thus $z^{\sigma}=0$. Hence $z(xy)=z^{\sigma}(xy)=0$ and $z(y)=z^{\sigma}(y)=0$. In the following, we assume $v \in_{G} D$. By our assumption, $LL(Z\beta/R\beta) \le m$ for all Brauer correspondent $\beta$ of $B$ in $C_{G}(v)$. Thus $z^{\sigma} \in RFC_{G}(v)$ and $z(xy)=z^{\sigma}(xy)=z^{\sigma}(y)=z(y)$ as required.

We next prove (2). Let $\mu : FC_{G}(u) \to F[C_{G}(u)/\left<u\right>]$ be the natural epimorphism and $m=LL(Z\bar{b})$. Then $\mu(J(Zb)^{m}) \subseteq J(Z\bar{b})^{m} = \{0\}$ and hence $J(Zb)^{mp^{n}} \subseteq (\ker \mu)^{p^{n}} = \{ (1-u)FC_{G}(u) \}^{p^{n}} = \{0\}$. Thus we have $LL(Zb) \le mp^{n}$. 
\end{proof}

The next lemma was proved originally by Passman \cite{P1}.

\begin{lem} \label{lem7} {\cite[Lemma 2]{KOS}} Let $B$ be a block of $FG$ with defect group $D$ and $P$ a subgroup of $Z(D)$. Then $Z_{\le P}(B) \cdot J(ZB)^{L} \subseteq Z_{< P}(B)$, where $L=LL(FP)$.
\end{lem}

By using the lemmas above we prove the following (cf. {\cite[Theorem 3]{KOS}}).

\begin{thm} \label{thm8} Let $B$ be a block of $FG$ with non-abelian defect group $D$ of order $p^{d}$. Then the following hold:
\begin{enumerate}
\item $LL(ZB) < 3p^{d-2}$ unless $D \simeq p^{1+2}_{+}, \ p \ge 5$; 
\item If $D$ has the form
\[ W(d) = < x, y, z \mid x^{p^{d-2}}=y^{p}=z^{p}=[x, y]=[x, z]=1, [y, z]=x^{p^{d-3}} > \]
then
\begin{equation*}
LL(ZB) \le \begin{cases}
                      4p-1 & (d=3) \\
                      2p^{2}+2p & (d=4) \\
                      p^{d-2}+2p(p-1)+1 & (d \ge 5);
                  \end{cases}    
\end{equation*}
\item If $B$ is controlled and $D \not\simeq W(d)$, then $LL(ZB) \le p^{d-2} + 3p^{d-3}+p-1$. 
\end{enumerate}
\end{thm}
\begin{proof}
(1) is a corollary to (2) and {\cite[Theorem 3]{KOS}}. We first suppose $D \simeq W(d), d \ge 5$. By Lemma \ref{lem6} (1), there exists $(u, b) \in \mathcal{S}$ such that $o(u)=p$ and $LL(ZB/RB) \le LL(Zb/Rb)$. If $u \notin Z(D)$, then $b$ has defect group $C_{D}(u) \simeq C_{p^{d-2}} \times C_{p}$. Thus we obtain from {\cite[Theorem 1]{KS}} that
\begin{align*}
 LL(ZB) \le LL(ZB/RB) + 1& \le LL(Zb/Rb)+1 \\
                                         & \le p^{d-2}+p-1 < p^{d-2}+2p(p-1)+1 
 \end{align*}
as $RB \cdot J(ZB)=\{0\}$. If $u \in Z(D)$, then $\bar{b}$ has defect group $D/\left<u\right> \simeq C_{p^{d-3}} \times C_{p} \times C_{p}$ and thus 
\begin{align*}
 LL(ZB) \le LL(Zb/Rb)+1 & \le LL(Zb)+1 \\
                                        & \le p \times LL(Z\bar{b}) +1 \\
                                        & \le p \times (p^{d-3}+2p-2)+1 = p^{d-2}+2p(p-1)+1
\end{align*}
by Lemma \ref{lem6} (2). This completes the proof of the last case in (2).

We next consider a block with defect group $D \simeq W(4)$. As $D$ has $p$-rank $2$ and $|D : Z(D)|=p^{2}$, all proper self-centralizing subgroups of $D$ are abelian of rank at most $2$. Hence $D$ has no $\mathcal{F}$-essential subgroup by {\cite[Proposition 6.11]{S1}} and $B$ is controlled by Alperin's fusion theorem. Therefore we can apply Lemma \ref{lem3} for the remaining cases. In particular, we may assume $p \neq 2$ and $e \le d-2$ in (3) by {\cite[Proposition 7]{KOS}}. By Lemma \ref{lem3}, there exists $(u, b) \in \mathcal{S} \backslash \mathcal{S}_{0}$ such that  $J(ZB)^{M} \subseteq Z_{\le Z(D)}(B)$, where $M = LL(Zb) \le p^{d-2}+p-1$ (remark that $b$ has defect group $C_{D}(u) < D$). If $D \not\simeq W(d)$, then $|Z(D)| \le p^{d-3}$ or $Z(D)$ is non-cyclic of order $p^{d-2}$ (see {\cite[Lemma 9]{KS}}). Thus it follows from Lemma \ref{lem7} that $Z_{\le Z(D)}(B) \cdot J(ZB)^{N} = \{0\}$, where
\begin{equation*}
N = \begin{cases}
        p^{2}+p+1 & (D \simeq W(4)) \\
        p^{d-3}+p-1+p^{d-3} + p^{d-4}+ \cdots + p+1 & (D \not\simeq W(d))
       \end{cases}
\end{equation*}
Hence we deduce the second case in (2) and (3) since $J(ZB)^{M+N} \subseteq Z_{\le Z(D)}(B) \cdot J(ZB)^{N} = \{0\}$ and $LL(ZB) \le M+N$. 
\end{proof}

We calculate one more case by the same way to Theorem \ref{thm8} (3). 

\begin{thm} Let $B$ be a controlled block of $FG$ with non-abelian defect group $D$ of order $p^{d}$ and exponent $p^{e}$. If $e \le d-3$, then $LL(ZB) < 6p^{d-3}$.
\end{thm}
\begin{proof}
We may assume $p \neq 2$. There exists $(u, b) \in \mathcal{S} \backslash \mathcal{S}_{0}$ such that $J(ZB)^{M} \subseteq Z_{\le Z(D)}(B)$, where $M = LL(Zb) < 3p^{d-3}$ by Theorem \ref{thm2} and \ref{thm8} (note that $C_{D}(u) \not\simeq p^{1+2}_{+}$). By our assumption, $|Z(D)| \le p^{d-3}$ or $Z(D)$ is non-cyclic of order $p^{d-2}$. Hence $Z_{\le Z(D)}(B) \cdot J(ZB)^{N} = \{0\}$, where $N=p^{d-3}+p-1+p^{d-3}+\dots +p+1$. Thereby we deduce $LL(ZB) \le M+N < 6p^{d-3}$.
\end{proof}

Finally, we state a conjecture that is based on the results in this section:

\begin{quote}
\textit{Every block $B$  with non-abelian defect group satisfies the following inequality ?}
\[ LL(ZB) < \frac{(d-e)(d-e+1)}{2} p^{e}  \]
\end{quote}

\end{document}